\documentclass[a4j, dvipdfmx]{amsart}
\usepackage{amsmath,amssymb}
\usepackage{amsthm}
\usepackage{graphicx}
\usepackage{bm}
\usepackage{tikz-cd}
\usepackage{mathtools}
\usepackage{comment}
\usepackage{bigdelim, multirow}
\usepackage{latexsym}
\usepackage{stmaryrd}
\usepackage{braket}

\newcommand{\Z}{\mathbb{Z}}
\newcommand{\R}{\mathbb{R}}
\newcommand{\C}{\mathbb{C}}
\newcommand{\Q}{\mathbb{Q}}

\newcommand\ch{\mathop{\mathrm{ch}}\nolimits}

\newcommand\RatCurves{\mathop{\mathrm{RatCurves}}\nolimits}
\def\i<#1>{\langle #1 \rangle}%
\def\l<#1>{\left\langle #1 \right\rangle}
\def\fr<#1/#2>{\frac{#1} {#2}}

\newcommand{\stirling}{\genfrac\{\}{0pt}{}}

\newtheorem{thm}{Dont use this}[section]
\newtheorem{theorem}[thm]{Theorem}
\newtheorem{proposition}[thm]{Proposition}
\newtheorem{corollary}[thm]{Corollary}
\newtheorem{lemma}[thm]{Lemma}

\newtheorem*{definition*}{Definition}

\newtheorem{definition and lemma}[thm]{Definition $\&$ Lemma}

\newtheorem*{remark*}{Remark}
\theoremstyle{definition}
\newtheorem{definition}[thm]{Definition}
\newtheorem{remark}[thm]{Remark}

\theoremstyle{definition}

\title{%
On a sufficient condition for a Fano manifold to be covered by rational $N$-folds
}
\author{%
Takahiro Nagaoka%
}
\address[Takahiro Nagaoka]{Department of Mathematics,
	Graduate School of Science,
	Kyoto University,
	Kyoto 606-8502, Japan}
\email{tnagaoka@math.kyoto-u.ac.jp}
\date{}

\begin{document}
\pagestyle{plain}

\begin{abstract}

In this paper, we prove a conjecture by T. Suzuki, which says if a smooth Fano manifold satisfies some positivity condition on its Chern character, then it can be covered by rational $N$-folds. We prove this conjecture by using purely combinatorial properties of Bernoulli numbers.  
\end{abstract} 
\maketitle
\section{Introduction}
For a Fano manifold $X$ i.e., a smooth projective variety whose anti-canonical line bundle $-K_X$ is ample, the following result by S. Mori is well-known (\cite{Mo}): every smooth Fano manifold $X$ is covered by rational curves. Thus, if we take a general point $x$, we can consider families of rational curves on $X$ through the point $x$. Among such families, one can consider {\it minimal} family $H_1$, which parametrizes rational curves whose intersection with $-K_X$ is minimal, and we write $X\vdash H_1$. Since one can obtain some geometric properties of $X$ from $H_1$, many authors studied $H_1$ extensively (\cite{Ko}, \cite{Ke}, \cite{Dr}, and so on). 

If $H_1$ is also a Fano manifold, one can consider minimal family $H_2$ of rational curves on $H_1$ through a general point. In \cite{Su}, Suzuki defined {\it higher order minimal families of rational curves} as a chain $X\vdash H_1\vdash \cdots\vdash H_N$ such that $H_i$ is Fano manifold for each $1\leq i\leq N-1$ and $H_i\vdash H_{i+1}$. He also introduced two crucial invariants $\underline{N}_X$ (resp. $\overline{N}_X$) which are the minimum (resp. maximum) length $N$ of higher order minimal families of rational curves $X\vdash H_1\vdash \cdots\vdash H_N$ such that $H_N$ is not a Fano manifold. Suzuki studied these invariants and proved the following higher dimensional generalization of S. Mori's result in the case $2\leq N\leq100$. Our main result is to prove this theorem for a general $N$ (cf. \cite[Conjecture 4.9]{Su}). 

\begin{theorem}\label{Main} Let $N$ be an integer $N\geq2$. Let $X$ be a Fano manifold with nef Chern characters $\ch_2(X), \ldots ,\ch_N(X)$. Assume that there is a family $X\vdash H_1$ of dimension at least $N^2-N-1$.
Then $\underline{N}_X\geq N$ holds. Moreover, if $H_1\ncong \mathbb{P}^d, Q^d$ and there exists families $H_1\vdash H_2\vdash \cdots \vdash H_{N-1}$ such that $H_2, \ldots ,H_{N-1}\ncong\mathbb{P}^d$, then $X$ is covered by rational $N$-folds.
\end{theorem}

Suzuki defined a sequence $\{b_{(i,j,k)} \in\Q\ | \ i, j\geq1, 1\leq k\leq i+j\}$ involving Bernoulli numbers (\cite[Notation 4.3]{Su}).
Then, in \cite[Notation 4.3, Proposition 4.4, Proposition 4.6(1)]{Su}, he showed $\ch_j(H_i)$ can be computed by using $b_{(i,j,k)}$ as the following;
\[\ch_j(H_i)=-i\fr<c_1(L_i)^j/j!>+\sum_{k=1}^{i-1}{b_{(i,j,k)}T^k(\ch_k(X))c_1(L_i)^j}+\sum_{k=i}^{i+j}{b_{(i,j,k)}T^i(\ch_k(X))c_1(L_i)^{i+j-k}}\]
(for the definition of ample line bundle $L_i$ on $H_i$ and linear map $T^m:N^r(H_i)_\R\to N^{r-m}(H_{i+m})_\R$, see \cite[Definition 2.12]{Su}).

From this expression, he showed that the positivity of $b_{(i,j,k)}$ ($1\leq i<N$, $j=1,2$, and $1\leq k\leq i+j$) implies Theorem \ref{Main} (cf. Theorem \ref{theorem:Suzuki}). 
Actually, he checked this positivity for $2\leq N\leq100$ by a computer and deduced the above theorem for $2\leq N\leq100$. We determine $b_{(i,j,k)}$ explicitly and prove its positivity for all $N\geq2$ as the following.  Consequently, we have the above theorem.
\begin{theorem}(cf. Corollary \ref{cor:explicitb} and Theorem \ref{thm:main})\\
For $i, j\geq1$ and $1\leq k\leq i+j$, we have the explicit formula for $b_{(i,j,k)}$ as the following.   
\begin{itemize}
\item[(1)]When $i\leq k\leq i+j$, 
\[b_{(i,j,k)}=\sum_{\substack{\ell_1+\cdots+\ell_i=i+j-k\\ \forall r, \ell_r\geq0}}{\fr<(-1)^{\ell_1}B_{\ell_1}/\ell_1!>\times\cdots\times\fr<(-1)^{\ell_i}B_{\ell_i}/\ell_i!>},\]
\item[(2)]When $1\leq k<i$, 
\[b_{(i,j,k)}=\fr<1/(-1)^{j}j!>\sum_{p=0}^{j}{(-1)^pp!\stirling{j}{p}\left(\sum_{\substack{\ell_1+\cdots+\ell_k=p+i-k\\ \forall r, \ell_r\geq0}}{\fr<1/(\ell_1+1)\cdots(\ell_k+1)>}\right)}.\]
\end{itemize}
where $\stirling{j}{p}$ are Stirling numbers of the second kind (see below Remark \ref{Daehee}), and $B_\ell$ are the $\ell$-th Bernoulli numbers. 

In particular, for $i\geq1$, $j=1,2$ and $1\leq k\leq i+j$, $b_{(i,j,k)}>0$ holds.
\end{theorem}

This paper is organized as follows. In Section 2, we explain some basic notions and review Suzuki's work.  In Section 3, for later convenience, we extend the range of the definition of $b_{(i,j,k)}$ and define $d_{(i,j,k)}$. In Section 4, we compute the generating function of $d_{(i,j,k)}$ and determine $b_{(i,j,k)}$ very explicitly. In Section 5, we prove $b_{(i,1,k)}>0$, $b_{(i,2,k)}>0$, and deduce Theorem\ref{Main}. 
\section*{Acknowledgements}
The author wishes to express his gratitude to Professor Taku Suzuki for letting me know his conjecture and giving me some useful comments. He is also grateful to Professor Norihiko Minami for letting him know he also proves the same conjecture by another method (one can read his preprint on arXiv in a near future).
\section{Suzuki's work on higher order minimal family of rational curves associated to Fano manifolds}
In this section, we review minimal family of rational curves of a Fano manifold and Suzuki's work (see \cite{Su} and the references given there). Let $X$ be a smooth Fano manifold of dimension $n$ and let $x\in X$ be a general point.
 
\begin{definition}
We define $\RatCurves^n(X, x)$ as the normalization of the scheme of all rational curves on $X$ passing through $x$ (see \cite[I, II]{Ko}). An irreducible component $H$ of $\RatCurves^n(X, x)$, which parameterize rational curves whose intersection number with $-K_X$ minimal, is called a {\it minimal family of rational curves through }$x$. In this case, we write $X\vdash H$, and $H$ is known to be smooth and proper (cf. \cite{Su}). 
\end{definition}

If we write $X\vdash H_1\vdash \cdots\vdash H_N$, it means each $H_{i-1}$ is a positive dimensional Fano manifold for each $1\leq i\leq N$ and satisfies $H_{i-1}\vdash H_i$, where $H_0:=X$.   

If $X\vdash H_1\vdash \cdots\vdash H_N$ is a chain of families, then we call this $N$ its {\it length}. 
\begin{definition}Let $X$ be a smooth algebraic variety. If $X$ is Fano manifold, we define $\underline{N}_X$ (resp. $\overline{N}_X$) as the minimum (resp. maximum) length of families $X\vdash H_1\vdash \cdots\vdash H_N$, where $H_N$ is not a Fano manifold. If $X$ is not a Fano manifold, we define $\underline{N}_X=\overline{N}_X=0$.
\end{definition} 

For a projective manifold $Z$, we denote by $N_r(Z)$ (resp. $N^r(Z)$) the group of cycles of dimension $r$ (resp. codimension $r$) on $Z$ modulo numerical equivalence. We set $N_r(Z)_{\R}:=N_r(Z)\otimes_{\Z}\R$ and $N^r(Z)_{\R}:=N^r(Z)\otimes_{\Z}\R$. 
\begin{definition}For a projective manifold $Z$, we say that $\alpha\in N^r(Z)_\R$ is {\it nef} if $\alpha\cdot\beta\geq0$ for every nonzero effective integral cycle $\beta\in N_r(Z)_\R$.
\end{definition}


To define $b_{(i,j,k)}$, we recall the Bernoulli numbers.
\begin{definition}(Bernoulli numbers)\\
We define the Bernoulli numbers $\{B_n\}_{n\geq0}$ by the following generating function;  
\[\fr<t/e^t-1>=\sum_{n=0}^{\infty}{B_n\fr<t^n/n!>}\]
\end{definition}
In \cite{Su}, for $i\geq1, j\geq1, 1\leq k\leq i+j$, $b_{(i,j,k)}\in\Q$ is defined by the following recurrence relation (cf. \cite{Su} Notation 4.3). 
\[b_{(1,j,k)}:=\fr<(-1)^{j+1-k}B_{j+1-k}/(j+1-k)!>\]
\begin{equation*}\tag{*}
b_{(i,j,k)}:=\sum_{m=0}^{\min\{j, i+j-k\}}{\fr<(-1)^mB_m/m!>b_{(i-1, j+1-m, k)}} \ \ (i\geq2)
\end{equation*} 
\begin{theorem}\label{theorem:Suzuki}(Conjecture 4.9 and Theorem 5.1 in  \cite{Su})\\
If $b_{(i,j,k)}>0$ for $i\geq1$, $j=1,2$, and $1\leq k\leq i+j$, then Theorem \ref{Main} holds. 
\end{theorem}
We will prove the positivity of $b_{(i,j,k)}$ for $i\geq1$, $j=1,2$, and $1\leq k\leq i+j$ by a combinatorial method and deduce the main theorem.  
\begin{theorem}
When $i\geq1$, $j=1,2$, $1\leq k\leq i+j$, 
\[b_{(i,j,k)}>0\]
In particular, Theorem \ref{Main} holds.  
\end{theorem}

\section{The extension of $b_{(i,j,k)}$}
For later convenience, in this section, we extend the definition of $b_{(i,j,k)} \ (i\geq1, j\geq1, 1\leq k\leq i+j)$ to $b_{(i,j,k)} \ (i\geq1, j\geq0, k\geq1)$ as the following.  
\begin{definition}For $i\geq1$, $j\geq0$, $k\geq1$, we define $b_{(i,j,k)}\in\Q$ by the following recurrence relation. 
\[b_{(1,j,k)}:=\begin{cases}
0&(k>1+j)\\
{}&{}\\
\fr<(-1)^{j+1-k}B_{j+1-k}/(j+1-k)!>&(k\leq1+j)
\end{cases}\]
\[b_{(i,j,k)}:=\sum_{m=0}^{j}{\fr<(-1)^mB_m/m!>b_{(i-1, j+1-m, k)}} \ \ (i\geq2)\] 
\end{definition}
\begin{lemma}\label{lem:zero}For our $b_{(i,j,k)}$ defined as above, we have
\[b_{(i,j,k)}=0\ \ (\text{if} \ k>i+j).\]
In particular, for $i\geq1$, $j\geq1$, and $1\leq k\leq i+j$, $b_{(i,j,k)}$ is exactly equal to $b_{(i,j,k)}$ defined in \cite{Su} $($cf. $(*))$. 
\end{lemma}
\begin{proof}
We prove the former claim by the induction on $i$. When $i=1$, the claim is clear. For $i\geq2$, by definition, we have
\[b_{(i,j,k)}:=\sum_{m=0}^{j}{\fr<(-1)^mB_m/m!>b_{(i-1, j+1-m, k)}}.\] 
Since $k>i+j\geq (i-1)+(j+1-m) \ \ (0\leq m\leq j)$, we have $b_{(i-1, j+1-m, k)}=0$ by the induction hypothesis. This proves $b_{(i,j,k)}=0\ \ (\text{if} \ k>i+j)$.

For the latter claim, we note $b_{(i,0,k)}$ doesn't affect on the recurrence relation for $i\geq2$. Thus, we only have to show if $i\geq2$, $k\leq i+j$, and $\min\{j, i+j-k\}=i+j-k$ (i.e., $2\leq i<k \leq i+j$), then 
\begin{equation*}\tag{**}
\sum_{m=0}^{i+j-k+1}{\fr<(-1)^mB_m/m!>b_{(i-1, j+1-m, k)}}=0.
\end{equation*}
Since for $i+j-k+1\leq m$, $(i-1)+(j+1-m)=i+j-m\leq k-1<k$, we have $b_{(i-1,j+1-m,k)}=0$ by the former claim. This proves (**). 
\end{proof}
For the later convenience, we set 
\[d_{(i,j,k)}:=(-1)^jj!b_{(i,j,k)}\]
and we will mainly consider the generating function and explicit formula for $d_{(i,j,k)}$ in the subsequent sections. 
By definition, the following is clear. 
\begin{lemma}
For $i\geq1$, $j\geq0$, $k\geq1$, $d_{(i,j,k)}$ satisfies the following recurrence relation; 
\[d_{(1,j,k)}:=\begin{cases}
0&(k>1+j)\\
{}&{}\\
(-1)^jj!\fr<(-1)^{j+1-k}B_{j+1-k}/(j+1-k)!>&(k\leq1+j)
\end{cases}\]
\[d_{(i,j,k)}:=-\fr<1/j+1>\sum_{m=0}^{j}{\binom{j+1}{m}B_md_{(i-1,j+1-m,k)}} \ \ \ (i\geq2)\]
\end{lemma}
\begin{remark}\label{rem:zero}
By definition and Lemma \ref{lem:zero}, if $k>i+j$, then $d_{(i,j,k)}=0$. 
\end{remark}

\section{The generating function of $d_{(i,j,k)}$}
In this section, we consider the generating function $D_{(i,k)}(t)$ of $d_{(i,j,k)}$;
\[D_{(i,k)}(t):=\sum_{j=0}^{\infty}{d_{(i,j,k)}\fr<t^j/j!>}.\] 

\begin{lemma}\label{lem:relation}
\[D_{(1,k)}(t)=\fr<(-t)^k/1-e^t>\]
\[D_{(i,k)}(t)=\fr<1/1-e^t>(-d_{(i-1,0,k)}+D_{(i-1,k)}(t)) \ \ (\text{if} \ i\geq2).\]
As a consequence, 
\[D_{(i,k)}(t)=\begin{cases}
\fr<(-t)^k/(1-e^t)^i> & (\text{if} \ k\geq i)\\
{}&{}\\
\left(\fr<1/1-e^t>\right)^{i-k}\left(\left(\fr<-t/1-e^t>\right)^k-\displaystyle\sum_{r=0}^{i-k-1}{d_{(r+k,0,k)}(1-e^t)^r}\right) & (\text{if} \ k<i)
\end{cases}
\]
\end{lemma}
\begin{proof}
Note
\[d_{(1,j,k)}:=\begin{cases}
0&(k>1+j)\\
{}&{}\\
(-1)^jj!\fr<(-1)^{j+1-k}B_{j+1-k}/(j+1-k)!>&(k\leq1+j)
\end{cases}.\]
Then, we can compute $D_{(1,k)}(t)$ as the following; 
\begin{align*}
D_{(1,k)}(t)&=\sum_{j=0}^{\infty}{d_{(1,j,k)}\fr<t^j/j!>}=\sum_{j=k-1}^{\infty}{d_{(1,j,k)}\fr<t^j/j!>}\\
&=\sum_{j=k-1}^{\infty}{(-1)^jj!\fr<(-1)^{j+1-k}B_{j+1-k}/(j+1-k)!>\fr<t^j/j!>}\\
&=\sum_{q=0}^{\infty}{(-1)^{q+k-1}\fr<(-1)^qB_q/q!>t^{q+k-1}}\\
&=(-t)^{k-1}\sum_{q=0}^{\infty}{B_q\fr<t^q/q!>}=\fr<(-t)^k/1-e^t>.
\end{align*}
On the other hand, for $i\geq2$, we note
\[d_{(i,j,k)}:=-\fr<1/j+1>\sum_{m=0}^{j}{\binom{j+1}{m}B_md_{(i-1,j+1-m,k)}} \ \ \ (i\geq2).\]
Then, we can compute $D_{(i,k)}(t)$ as the following; 
\begin{align*}
D_{(i,k)}(t)&=-\sum_{j=0}^{\infty}{\fr<1/j+1>\left(\sum_{m=0}^j{\binom{j+1}{m}B_md_{(i-1,j+1-m,k)}}\right)\fr<t^j/j!>}\\
&=-\left(\sum_{j=0}^{\infty}{B_j\fr<t^j/j!>}\right)\left(\sum_{j=0}^{\infty}{d_{(i-1,j+1,k)}\fr<t^j/(j+1)!>}\right)\\
&=-\left(\fr<t/e^t-1>\right)\left(\fr<1/t>\sum_{j=0}^{\infty}{d_{(i-1,j+1,k)}\fr<t^{j+1}/(j+1)!>}\right)\\
&=\fr<1/1-e^t>\left(D_{(i-1,k)}(t)-d_{(i-1,0,k)}\right).
\end{align*}
In addition, we use the recurrence relation above repeatedly; then $D_{(i,k)}(t)$ can be computed as follows. 
\begin{align*}
D_{(i,k)}(t)&=\fr<1/1-e^t>(-d_{(i-1,0,k)}+D_{(i-1,k)}(t))\\
&=\fr<1/1-e^t>\left\{-d_{(i-1,0,k)}+\fr<1/1-e^t>\left\{-d_{(i-2,0,k)}+\cdots+\fr<1/1-e^t>\left\{-d_{(1,0,k)}+D_{(1,k)}(t)\right\}\cdots\right\}\right\}\\
&=\fr<-d_{(i-1,0,k)}/1-e^t>+\fr<-d_{(i-2,0,k)}/(1-e^t)^2>+\cdots+\fr<-d_{(1,0,k)}/(1-e^t)^{i-1}>+\fr<D_{(1,k)}(t)/(1-e^t)^{i-1}>
\end{align*}
Since by Remark \ref{rem:zero}, $d_{(k-1,0,k)}=d_{(k-2,0,k)}=\cdots=d_{(1,0,k)}=0$, we obtain the following. 

If $k\geq i$, then
\[D_{(i,k)}(t)=\fr<D_{(1,k)}(t)/(1-e^t)^{i-1}>=\fr<(-t)^k/(1-e^t)^i>.\]

If $k<i$, then
\begin{align*}
D_{(i,k)}(t)&=\fr<-d_{(i-1,0,k)}/1-e^t>+\fr<-d_{(i-2,0,k)}/(1-e^t)^2>+\cdots+\fr<-d_{(k,0,k)}/(1-e^t)^{i-k}>+\fr<D_{(1,k)}(t)/(1-e^t)^{i-1}>\\
&=\fr<-d_{(i-1,0,k)}/1-e^t>+\fr<-d_{(i-2,0,k)}/(1-e^t)^2>+\cdots+\fr<-d_{(k,0,k)}/(1-e^t)^{i-k}>+\fr<(-t)^k/(1-e^t)^{i}>\\
&=\left(\fr<1/1-e^t>\right)^{i-k}\left(\left(\fr<-t/1-e^t>\right)^k-\sum_{r=0}^{i-k-1}{d_{(r+k,0,k)}(1-e^t)^r}\right).
\end{align*}
These complete the proof. 
\end{proof}

For $i>k$, by using the lemma above and the fact $t=\log(1-(1-e^t))$, we have
\begin{align*}
D_{(i,k)}(t)&=\left(\fr<1/1-e^t>\right)^{i-k}\left(\left(\fr<-t/1-e^t>\right)^k-\sum_{r=0}^{i-k-1}{d_{(r+k,0,k)}(1-e^t)^r}\right)\\
&=\left(\fr<1/1-e^t>\right)^{i-k}\left(\left(-\fr<\log(1-(1-e^t))/1-e^t>\right)^k-\sum_{r=0}^{i-k-1}{d_{(r+k,0,k)}(1-e^t)^r}\right)\\
&=\fr<f_k(1-e^t)-\displaystyle\sum_{r=0}^{i-k-1}{d_{(r+k,0,k)}(1-e^t)^r}/(1-e^t)^{i-k}>,\\
\end{align*}
where $f_k(s):=\left(-\fr<\log(1-s)/s>\right)^k$ (Note $f_k(1-e^t)$ is well-defined since $1-e^t$ doesn't have constant term). 
In this setting, as shown in the following general lemma, we can prove 
\[\sum_{q=0}^{\infty}{d_{(q+k,0,k)}s^{q}}=\left(-\fr<\log(1-s)/s>\right)^k.\]

\begin{lemma}Let $f(s):=\displaystyle\sum_{n=0}^{\infty}{c_ns^n}$ be the generating function of sequence $\{c_n\}_{n\geq0}$. For a finite sequence $d_0, d_1, \ldots, d_{\ell-1}$, we consider the following formal Laurent series $F_{\ell}(t)\in\C[[t]][\fr<1/t>]$; 
\[F_\ell(t):=\fr<f(1-e^t)-d_0-d_1(1-e^t)-\cdots-d_{\ell-1}(1-e^t)^{\ell-1}/(1-e^t)^{\ell}>.\]
The followings are equivalent; 
\begin{itemize}
\item[(i)]$F_\ell(t)\in\C[[t]]$ i.e., $F_\ell(t)$ doesn't have any negative term. 
\item[(ii)]$c_n=d_n \ \ (0\leq n\leq\ell-1).$
\end{itemize}
\end{lemma}
\begin{proof}
To find the negative leading term of $F_{\ell}(t)$ (i.e., $a_0(\neq0)$ if $F_{m}(t)=\fr<a_0/t^{m}>+\fr<a_1/t^{m-1}>+\cdots$ for some $m>0$, and 0 otherwise), we can compute:  
\begin{align*}
&F_\ell(t)=\fr<f(1-e^t)-d_0-d_1(1-e^t)-\cdots-d_{\ell-1}(1-e^t)^{\ell-1}/(1-e^t)^{\ell}>\\
&=\fr<(1-e^t)^{\ell}\left(\displaystyle\sum_{n=0}^{\infty}{c_{n+\ell}(1-e^t)^n}\right)+\displaystyle\sum_{r=0}^{\ell-1}{(c_r-d_r)(1-e^t)^r}/(1-e^t)^{\ell}>\\
&=\left(\displaystyle\sum_{n=0}^{\infty}{c_{n+\ell}(1-e^t)^n}\right)+\left(\fr<t/1-e^t>\right)^\ell\left(\fr<\displaystyle\sum_{r=0}^{\ell-1}{(c_r-d_r)(1-e^t)^r}/t^\ell>\right)
\end{align*}
Since we have
\[\left(\fr<t/1-e^t>\right)^\ell=\left(\sum_{n=0}^{\infty}{B_n\fr<t^n/n!>}\right)^\ell=1+t(\cdots) \]
and $1-e^t=-t-\fr<t^2/2!>-\cdots$, the negative leading term of $F_\ell(t)$ is given by $c_{i_0}-d_{i_0}$, where $i_0:=\min\{ i\in \{0,1,\ldots, \ell-1\} \ | \ c_i-d_i\neq0\}$. 
In particular, this means the equivalence of condition (i) and (ii). 
\end{proof}
\begin{corollary}\label{cor:log}(The generating function and an explicit formula for $d_{(q+k,0,k)}$)\\
\[\sum_{q=0}^{\infty}{d_{(q+k,0,k)}s^{q}}=\left(-\fr<\log(1-s)/s>\right)^k\]
Moreover, 
\[d_{(q+k,0,k)}=\sum_{\substack{\ell_1+\cdots+\ell_k=q \\ \forall r, \ell_r\geq0}}{\fr<1/(\ell_1+1)\cdots(\ell_k+1)>}.\]
\end{corollary}
\begin{proof}
The first claim follows from Lemma \ref{lem:relation} if we take $f$ and $d_r$ as the following; 
\[f(s)=\left(-\fr<\log(1-s)/s>\right)^{k}, \ \ \ \ \ d_r=d_{(k+r,0,k)} \ (0\leq r\leq \ell-1:=i-k-1).\]
For the second claim, we note 
\[-\fr<\log(1-s)/s>=\sum_{q=0}^{\infty}{\fr<s^{q}/q+1>}=\sum_{q=0}^{\infty}{\fr<q!/q+1>\fr<s^q/q!>}.\]
Then, 
\begin{align*}
\sum_{q=0}^{\infty}{d_{(q+k,0,k)}s^{q}}&=\left(-\fr<\log(1-s)/s>\right)^k\\
&=\sum_{q=0}^{\infty}{\left(\sum_{\ell_1+\cdots+\ell_k=q}{\binom{q}{\ell_1, \ldots, \ell_k}\fr<\ell_1!/(\ell_1+1)>\cdots\fr<\ell_k!/(\ell_k+1)>}\right)\fr<s^q/q!>}\\
&=\sum_{q=0}^{\infty}{\left(\sum_{\ell_1+\cdots+\ell_k=q}{\fr<1/(\ell_1+1)\cdots(\ell_k+1)>}\right)s^q},
\end{align*}
where
\[\binom{n}{\ell_1, \ldots ,\ell_i}:=\fr<n!/\ell_1!\cdots\ell_i!>\]
is the multinomial coefficient.  
This completes the proof. 
\end{proof}

\begin{remark}\label{Daehee}
$d_{(q+k,0,k)}$ can be expressed by the {\it higher order Daehee number} $D_{q}^{(k)}$ as the following. 
\[d_{(q+k,0,k)}=(-1)^q\fr<D_q^{(k)}/q!>,\]
where $D_q^{(k)}$ is defined by the following generating function (cf. \cite{Kim});  
\[\left(\fr<\log(1+t)/t>\right)^k=\sum_{q=0}^{\infty}{D_q^{(k)}\fr<t^q/q!>}\]
\end{remark}
Using this corollary and Lemma \ref{lem:relation}, we can compute the explicit formula for $d_{(i,j,k)}$ as the following. 
To represent it explicitly, we recall that the Stirling numbers of the second kind $\stirling{j}{p} \ (j\geq0, \ 0\leq p\leq j)$ is defined  by the following generating function;  
\[\fr<(e^t-1)^p/p!>=\sum_{j=p}^{\infty}{\stirling{j}{p}\fr<t^j/j!>}\]

\begin{proposition}
\begin{itemize}
\item[(1)]When $k\geq i$, 
\[d_{(i,j,k)}=\begin{cases}
0&(\text{if} \ k>i+j)\\
(-1)^{j}j!\displaystyle\sum_{\substack{\ell_1+\cdots+\ell_i=i+j-k\\ \forall r, \ell_r\geq0}}{\fr<(-1)^{\ell_1}B_{\ell_1}/\ell_1!>\times\cdots\times\fr<(-1)^{\ell_i}B_{\ell_i}/\ell_i!>}&(\text{if} \ i\leq k\leq i+j)
\end{cases}\]
\item[(2)]When $k<i$,  
\begin{align*}
d_{(i,j,k)}&=\sum_{p=0}^{j}{(-1)^pp!\stirling{j}{p}\left(\sum_{\substack{\ell_1+\cdots\ell_k=p+i-k\\ \forall r, \ell_r\geq0}}{\fr<1/(\ell_1+1)\cdots(\ell_k+1)>}\right)}
\end{align*}
\end{itemize}
\end{proposition}
\begin{proof}
(1) As noted in Lemma \ref{lem:relation}, for $k\geq i$ we have  
\[D_{(i,k)}(t):=\sum_{j=0}^{\infty}{d_{(i,j,k)}\fr<t^j/j!>}=\fr<(-t)^k/(1-e^t)^i>=(-t)^{k-i}\left(\fr<t/e^t-1>\right)^i.\]
Thus, we can compute the explicit formula for $d_{(i,j,k)}$ as the following 
\begin{align*}
D_{(i,k)}(t)&=(-t)^{k-i}\left(\sum_{\ell=0}^{\infty}{B_{\ell}\fr<t^{\ell}/\ell!>}\right)^i\\
&=(-t)^{k-i}\sum_{n=0}^{\infty}\left(\sum_{\ell_1+\cdots+\ell_i=n}{\binom{n}{\ell_1, \ldots ,\ell_i}B_{\ell_1}\cdots B_{\ell_i}}\right)\fr<t^n/n!>\\
&=(-1)^{k-i}\sum_{n=0}^{\infty}\left(\sum_{\ell_1+\cdots+\ell_i=n}{\binom{n}{\ell_1, \ldots ,\ell_i}B_{\ell_1}\cdots B_{\ell_i}}\right)\fr<t^{n+k-i}/n!>\\
&=(-1)^{k-i}\sum_{j=k-i}^{\infty}{\left(\sum_{\ell_1+\cdots+\ell_i=j-(k-i)}{\binom{j-(k-i)}{\ell_1, \ldots, \ell_i}B_{\ell_1}\cdots B_{\ell_i}}\right)\fr<t^j/(j-(k-i))!>}\\
&=(-1)^{k-i}\sum_{j=k-i}^{\infty}{j!\left(\sum_{\ell_1+\cdots+\ell_i=j-(k-i)}{\fr<B_{\ell_1}/\ell_1!>\times\cdots\times\fr<B_{\ell_i}/\ell_i!>}\right)\fr<t^j/j!>}\\
&=\sum_{j=k-i}^{\infty}{(-1)^jj!\left(\sum_{\ell_1+\cdots+\ell_i=j-(k-i)}{\fr<(-1)^{\ell_1}B_{\ell_1}/\ell_1!>\times\cdots\times\fr<(-1)^{\ell_i}B_{\ell_i}/\ell_i!>}\right)\fr<t^j/j!>}
\end{align*}
This completes the proof. 

(2)\ For $k<i$, by Lemma \ref{lem:relation} and Corollary \ref{cor:log}, we have 
\begin{align*}
D_{(i,k)}(t)&=\left(\fr<1/1-e^t>\right)^{i-k}\left(\left(\sum_{q=0}^{\infty}{d_{(q+k,0,k)}(1-e^t)^{q}}\right)-\sum_{r=0}^{i-k-1}{d_{(r+k,0,k)}(1-e^t)^r}\right)\\
&=\left(\fr<1/1-e^t>\right)^{i-k}\left(\sum_{q=i-k}^{\infty}{d_{(q+k,0,k)}(1-e^t)^{q}}\right)\\
&=\left(\fr<1/1-e^t>\right)^{i-k}(1-e^t)^{i-k}\left(\sum_{p=0}^{\infty}{d_{(p+i,0,k)}(1-e^t)^p}\right)\\
&=\sum_{p=0}^{\infty}{d_{(p+i,0,k)}(-1)^pp!\times\fr<(e^t-1)^p/p!>}\\
&=\sum_{p=0}^{\infty}{\left(d_{(p+i,0,k)}(-1)^pp!\sum_{j=p}^{\infty}{\stirling{j}{p}}\fr<t^j/j!>\right)}\\
&=\sum_{j=0}^{\infty}{\left(\sum_{p=0}^{j}{d_{(p+i,0,k)}(-1)^pp!\stirling{j}{p}}\right)\fr<t^j/j!>}.
\end{align*}
Thus we obtain 
\[d_{(i,j,k)}=\sum_{p=0}^{j}{d_{(p+i,0,k)}(-1)^pp!\stirling{j}{p}}.\]
Finally, by Corollary \ref{cor:log}, we obtain the following; 

\[d_{(i,j,k)}=\sum_{p=0}^{j}{\left(\sum_{\ell_1+\cdots\ell_k=p+i-k}{\fr<1/(\ell_1+1)\cdots(\ell_k+1)>}\right)(-1)^pp!\stirling{j}{p}}\]
\end{proof}
\begin{remark}For $i\leq k\leq i+j$, $d_{(i,j,k)}$ can be expressed by the {\it higher order Bernoulli number} $B_n^{(i)}$ as the following (cf. \cite{Kim}).
\[d_{(i,j,k)}=(-1)^{k-i}B_{i+j-k}^{(i)}\fr<j!/(i+j-k)!>,\]
where $B_n^{(i)}$ is defined by the following generating function. 
\[\left(\fr<t/e^t-1>\right)^i=\sum_{n=0}^{\infty}{B_n^{(i)}\fr<t^n/n!>}\] 
\end{remark}
From the above Proposition, we obtainan explicit formula for $b_{(i,j,k)}:=\fr<d_{(i,j,k)}/(-1)^jj!>$. 
\begin{corollary}\label{cor:explicitb}(An explicit formula for $b_{(i,j,k)}$)\\
For $i, j\geq1$ and $1\leq k\leq i+j$, we have an explicit formula for $b_{(i,j,k)}$ as the following;   
\begin{itemize}
\item[(1)]When $i\leq k\leq i+j$, 
\[b_{(i,j,k)}=\sum_{\substack{\ell_1+\cdots+\ell_i=i+j-k\\ \forall r, \ell_r\geq0}}{\fr<(-1)^{\ell_1}B_{\ell_1}/\ell_1!>\times\cdots\times\fr<(-1)^{\ell_i}B_{\ell_i}/\ell_i!>}\]
\item[(2)]When $k<i$, 
\[b_{(i,j,k)}=\fr<1/(-1)^{j}j!>\sum_{p=0}^{j}{(-1)^pp!\stirling{j}{p}\left(\sum_{\substack{\ell_1+\cdots+\ell_k=p+i-k\\ \forall r, \ell_r\geq0}}{\fr<1/(\ell_1+1)\cdots(\ell_k+1)>}\right)}.\]
\end{itemize}
\end{corollary}

In the next section, we will prove the main theorem.  
\section{The proof of $b_{(i,1,k)}, \ b_{(i,2,k)}>0$}
In this section, we complete the proof of the main result.  
\begin{theorem}\label{thm:main}
When $i\geq1$, $j=1,2$, and $1\leq k\leq i+j$, we have
\[b_{(i,j,k)}>0\]
\end{theorem}
\begin{proof}
As in the last corollary, we divide the proof into two cases: (1) $i\leq k$ and (2) $k<i$. 

Case (1) : \ $i\leq k\leq i+j$.

By the above corollary, we know
\[b_{(i,j,k)}=\sum_{\ell_1+\cdots+\ell_i=i+j-k}{\fr<(-1)^{\ell_1}B_{\ell_1}/\ell_1!>\times\cdots\times\fr<(-1)^{\ell_i}B_{\ell_i}/\ell_i!>}.\]
Since $0\leq i+j-k\leq j$ and $j=1 \ \text{or} \ 2$, each of all $\ell_1, \ldots ,\ell_i$ is 0, 1, or 2 in the sum part of the above equation.  Now, by definition, we note $B_0=1$, $B_1=-\fr<1/2>$, $B_2=\fr<1/6>$. In particular, $(-1)^\ell B_\ell>0$ \ (for $\ell=0,1,2$). 
This proves $b_{(i,j,k)}>0$ (for $j=1,2$). 

Case (2) : \ $k<i$.

By the above corollary, we know
\[b_{(i,j,k)}=\fr<1/(-1)^{j}j!>\sum_{p=0}^{j}{(-1)^pp!\stirling{j}{p}\left(\sum_{\ell_1+\cdots+\ell_k=p+i-k}{\fr<1/(\ell_1+1)\cdots(\ell_k+1)>}\right)}.\]
By definition, we note $\stirling{1}{0}=\stirling{2}{0}=0$, $\stirling{1}{1}=\stirling{2}{1}=\stirling{2}{2}=1$. First, when $j=1$, 
\begin{align*}
b_{(i,1,k)}&=\fr<1/(-1)1!>(-1)1!\stirling{1}{1}\left(\sum_{\ell_1+\cdots+\ell_k=1+i-k}{\fr<1/(\ell_1+1)\cdots(\ell_k+1)>}\right)\\
&=\sum_{\ell_1+\cdots+\ell_k=1+i-k}{\fr<1/(\ell_1+1)\cdots(\ell_k+1)>} >0 .
\end{align*}
Finally, when $j=2$, 
\begin{align*}
b_{(i,2,k)}&=\fr<1/(-1)^22!>\left\{\begin{array}{l}(-1)1!\stirling{2}{1}\left(\displaystyle\sum_{\ell_1+\cdots+\ell_k=1+i-k}{\fr<1/(\ell_1+1)\cdots(\ell_k+1)>}\right)\\+(-1)^22!\stirling{2}{2}\left(\displaystyle\sum_{\ell_1+\cdots+\ell_k=2+i-k}{\fr<1/(\ell_1+1)\cdots(\ell_k+1)>}\right)\end{array}\right\}\\
&=\fr<1/2>\left\{2\sum_{\ell_1+\cdots+\ell_k=1+i-k}{\fr<1/(\ell_1+1)\cdots(\ell_k+1)>}-\sum_{\ell_1+\cdots+\ell_k=2+i-k}{\fr<1/(\ell_1+1)\cdots(\ell_k+1)>}\right\}\\
&=\fr<1/2>\sum_{\ell_1+\cdots+\ell_k=1+i-k}{\left\{\left(\sum_{s=1}^k{\fr<2/(\ell_1+1)\cdots(\ell_s+2)\cdots(\ell_k+1)>}\right)-\fr<1/(\ell_1+1)\cdots(\ell_s+1)\cdots(\ell_k+1)>\right\}}\\
&=\fr<1/2>\sum_{\ell_1+\cdots+\ell_k=1+i-k}{\sum_{s=1}^k{\left(\fr<2/(\ell_1+1)\cdots(\ell_s+2)\cdots(\ell_k+1)>-\fr<1/k>\fr<1/(\ell_1+1)\cdots(\ell_s+1)\cdots(\ell_k+1)>\right)}}\\
&=\fr<1/2>\sum_{\ell_1+\cdots+\ell_k=1+i-k}{\sum_{s=1}^k{\left\{\left(\fr<2/(\ell_s+2)>-\fr<1/k>\fr<1/(\ell_s+1)>\right)\left(\prod_{t\neq s}\fr<1/(\ell_t+1)>\right)\right\}}}\\
\end{align*}
Moreover, we have
\begin{align*}
\fr<2/(\ell_s+2)>-\fr<1/k>\fr<1/(\ell_s+1)>&=\fr<1/k>\left(\fr<2k/\ell_s+2>-\fr<1/\ell_s+1>\right)\\
&=\fr<1/k>\cdot\fr<2k(\ell_s+1)-(\ell_s+1)-1/(\ell_s+2)(\ell_s+1)>\\
&=\fr<1/k>\cdot\fr<(2k-1)(\ell_s+1)-1/(\ell_s+2)(\ell_s+1)>.
\end{align*}
Clearly, this shows $b_{(i,2,k)}>0$ if $k>1$. For $k=1$, we can compute directly $b_{(i,2,k)}=\fr<i/2(i+2)(i+1)>>0$ from the above equation. This completes the proof. 
\end{proof}


\end{document}